\documentclass[11pt,a4paper]{article}

\usepackage[OT1]{fontenc}
\usepackage{lmodern}
\usepackage[english]{babel}
\usepackage[utf8x]{inputenc}
\usepackage[sc]{mathpazo}
\usepackage{amsmath,amssymb,amsfonts,mathrsfs,amstext}
\usepackage[thmmarks]{ntheorem}
\usepackage{wrapfig}
\usepackage{nicefrac}
\usepackage{tabularx}

\usepackage{varioref}
\usepackage{datetime}
\usepackage{mathtools}
\usepackage{array}
\usepackage{listings}
\usepackage[linkcolor=black,colorlinks=true,citecolor=black,filecolor=black]{hyperref}

\usepackage{algpseudocode}
\usepackage{algorithm}

\usepackage{color}
\usepackage{caption}
\usepackage{subcaption}
\usepackage[noabbrev]{cleveref}
\usepackage{csquotes}
\usepackage{bm}
\usepackage{mathtools}
\usepackage{booktabs}

\numberwithin{equation}{section}
\newtheorem{theorem}{Theorem}[section]

\newtheorem{remark}[theorem]{Remark}

\newtheorem{lemma}[theorem]{Lemma}

\newtheorem{proof}{Proof}

\theoremstyle{nonumberplain}
\theorembodyfont{\normalfont}
\theoremsymbol{\ensuremath{\blacksquare}}

\renewcommand{\epsilon}{\ensuremath\varepsilon}

\renewcommand{\phi}{\ensuremath{\varphi}}

\newcommand\mathmidscript[1]{\vcenter{\hbox{$\scriptstyle #1$}}}

\newcommand{\pvint}{\ensuremath{\;\mathmidscript{\raisebox{-0.9ex}{\tiny{p.\,v.\;}}}\!\!\!\!\!\!\!\!\int}}

\DeclareMathAlphabet{\mathpzc}{OT1}{pzc}{m}{it}

\newcommand*\Laplace{\mathop{}\!\mathbin\bigtriangleup}

\newcommand{\DEF}{\coloneqq}

\newcommand{\RR}{\mathbb{R}}
\newcommand{\CC}{\mathbb{C}}
\newcommand{\NN}{\mathbb{N}}

\newcommand{\Om}{\Omega}

\newcommand{\del}{\partial}

\newcommand{\MID}{\!\! \mid\!}

\newcommand{\cC}{\mathcal{C}}

\newcommand{\OO}{\mathcal{O}}

\newcommand{\Reu}{\mathrm{R}}

\newcommand{\intd}{\mathrm{d}}

\newcommand{\TransT}{\mathrm{T}}

\makeatletter
\def\moverlay{\mathpalette\mov@rlay}
\def\mov@rlay#1#2{\leavevmode\vtop{%
   \baselineskip\z@skip \lineskiplimit-\maxdimen
   \ialign{\hfil$\m@th#1##$\hfil\cr#2\crcr}}}
\newcommand{\charfusion}[3][\mathord]{
    #1{\ifx#1\mathop\vphantom{#2}\fi
        \mathpalette\mov@rlay{#2\cr#3}
      }
    \ifx#1\mathop\expandafter\displaylimits\fi}
\makeatother

\newcommand{\Gk}{\Gamma^k}

\newcommand*\imagi{\mathrm{i}\,}

\newcommand{\compactvec}[1]{\ensuremath
    \big(\begin{smallmatrix}#1\end{smallmatrix}\big)%
}

\newcommand{\Omr}{{\Omega}_r}
\newcommand{\OmR}{{\Omega}_R}
\newcommand{\NOk}{\mathrm{N}_{\Omega}^k}
\newcommand{\NOrk}{\mathrm{N}_{\Omr}^k}
\newcommand{\NORk}{\mathrm{N}_{\OmR}^k}
\newcommand{\ROk}{\Reu_{\Omega}^k}

\newcommand{\Hcirc}{\mathring{\mathcal{H}}}

\title{Optimal design of optical analog solvers of linear systems}
\date{}
\author{Kthim Imeri\thanks{\footnotesize Department of Mathematics, ETH Z\"urich, R\"amistrasse 101, CH-8092 Z\"urich, Switzerland (kthim.imeri@sam.math.ethz.ch).} 
}

\begin{document}
	\maketitle

\begin{abstract}
In this paper, given a linear system of equations $\mathbf{A}\, \mathbf{x}= \mathbf{b}$, we are finding locations in the plane to place objects such that sending waves from the source points and gathering them at the receiving points solves that linear system of equations. The ultimate goal is to have a fast physical method for solving linear systems. The issue discussed in this paper is to apply a fast and accurate algorithm to find the optimal locations of
 the scattering objects. We tackle this issue by using asymptotic expansions for the solution of the underlying partial differential equation. This also yields a potentially faster algorithm than the classical BEM for finding solutions to the Helmholtz equation. 
\end{abstract}

\def\keywords2{\vspace{.5em}{\textbf{Mathematics Subject Classification
(MSC2000).}~\,\relax}}
\def\endkeywords2{\par}
\keywords2{35C20, 78A46}

\def\keywords{\vspace{.5em}{\textbf{ Keywords.}~\,\relax}}
\def\endkeywords{\par}
\keywords{optical solver of linear systems, scattering of waves, Neumann functions}

\section{Introduction}\label{Ch:Introduction}

In physical problems such as reflection of light in a three dimensional environment, aircraft simulations, or image recognition, we 
are searching for methods to numerically solve linear systems of equations of the form $\mathbf{A}\, \mathbf{x}= \mathbf{b}$. Recently, optical analog computing has been introduced as an alternative paradigm to classical computational linear algebra in order to contribute to computing technology.  In  \cite{EnghetaPaper}, 
the authors design a two dimensional structure with a physical property, which allows for solving a predetermined linear system of equations. In general, structures with such favourable physical properties are called meta-structures or meta-surfaces and are under very active research \cite{ourGradSurfpaper,HRSuperLens, Lin298}. To be precise in their set-up, sending specific waves across the meta-structure modifies those waves, such that they represent the solution $\mathbf{x}$ to the problem. 

One issue with this method is to quickly find the accurate structure for a given matrix $\mathbf{A}$. In \cite{EnghetaPaper} the authors use physics software to gradually form such a structure. Here we demonstrate another method, which relies on using asymptotic expansions of solutions to partial differential equation. Such expansions have already been studied in different papers \cite{ourzarembapaper, oursteklovpaper, FWMSP1, FWMSP2}. With that tool we can position and scale objects, on which the waves scatter, such that the resulting structure satisfies the desired requirements.

In the process of developing these asymptotic formulas, we have realized that we can compute the scattered wave using a method which is similar to the explicit Euler scheme. There we can numerically compute the solution to an ordinary differential equation by successively progressing in time with small time-steps until we reach the desired time. With our method we solve the partial differential equation around an obstacle by progressing the object-radius with small radius increments until we reach the full extent. We present the numerical application of that method on a circular domain.

This paper is organized as follows. In Section \ref{sec:Prelim} we model the mathematical foundation for the underlying physical problem and define the fundamental partial differential equations for the asymptotic expansions. This leads us to the definition for the Neumann function on the outside. We then explain the connection of the Neumann function and the linear system of equations. In Section \ref{Sec:AsymptoticFormula} we prove the asymptotic formulas concerning the Neumann function. There we discover special singular behaviours, which are essential to prove the asymptotic expansions. In Section \ref{sec:NumImplTest}, we first show the method to solve for the wave by increasing the radius by small steps and discuss the numerical error. Afterwards, we explain how we numerically build the meta-structure to solve the linear system of equations and discuss how well it operates. In Section \ref{sec:ConcludingRemarks}, we conclude the paper with final considerations, open questions and possible future research directions. In the appendix we provide an interesting proof of a technical result and a modification of the trapezoidal rule, when we apply a logarithmic singularity.

\section{Preliminaries}\label{sec:Prelim}
Let $k\in (0,\infty)$ and let $\Omega$ be a finite union of disjoint, non-touching, simply connected, bounded and open $\cC^2$-domains in $\RR^2_+ = \{ x\in \RR^2 \mid x_2>0\}$. Let $(z_j)_{j=1}^N$ be $N\in\NN$ source points on the horizontal axis $\Lambda = \{x\in \RR^2 \mid x_1\in (0,1) \text{ and } x_2=0\}$.  We have $N$ functions $u_j: \RR^2\rightarrow \CC$, for $j=1, \ldots, N$, which solve the following partial differential equation:
\begin{align}\label{PDE:uj}
	\left\{ 
	\begin{aligned}
		 \left( \Laplace + k^2  \right) u_j(x) &= I_j \,\delta_{z_j}(x) \quad &&\text{in } \RR^2_+\setminus \overline{\Om}\,, \\
		 \del_{\nu_x} u_j(x) \MID_+\! &= 0 \quad &&\text{on } \del \Omega \,,\\
		 u_j(x) &= 0 \quad &&\text{on } \del\RR^2_+\setminus \Lambda \,,\\
		 \del_{x_2} \, u_j(x) &= 0 \quad &&\text{on } \Lambda \,,\\
		 \Big({\frac{\del }{\del |x|}}-\imagi k\Big)\,u_j(x)&\rightarrow 0 \quad &&\text{for } |x|\rightarrow \infty\,,
	\end{aligned}
	\right.
\end{align}
\begin{wrapfigure}{r}{0.5\textwidth}
  \begin{center}
    \includegraphics[width=0.48\textwidth]{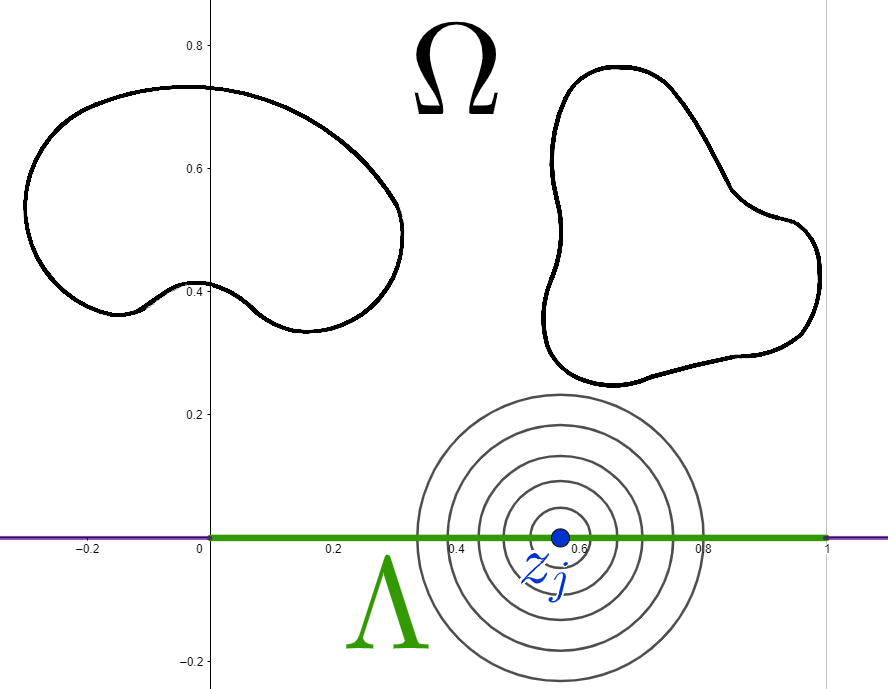}
  \end{center}
  \caption{This is the geometric set-up. In black we have the domain $\Omega$, in green the line segment $\Lambda$ and in blue the source point $z_j$. In violet, the absorbing layer is depicted.}
\end{wrapfigure}
where $\cdot\MID_+\!$ denotes the limit from the outside of $\Omega$ to $\del \Om$, and $\del_{\nu_x}$ denotes the outside normal derivative on $\del \Omega$. $\delta_{z_j}$ denotes the Dirac delta function at point $z_j$ and $I_j\in\CC$ denotes the intensity at the source $z_j$. The first condition is the Helmholtz equation which represents the time-independent wave equation, and arises from the wave equation using the Fourier transform in time on the wave equation. The second condition is known as the Neumann condition and models a material with a high electrical impedance. The third and fourth conditions model an absorbing layer on $\del\RR^2_+\setminus \Lambda$ and a Neumann condition on $\Lambda$. The fifth condition is known as the Sommerfeld radiation condition, and originates from a physical constraint for the behaviour of an outgoing wave.

We define $\Gamma^k$ to be the fundamental solution to the Helmholtz equation, that is $\Gamma^k$ solves PDE (\ref{PDE:uj}) without the Neumann boundary condition and a source at the origin. Furthermore we define $\Gamma^k(z,x) = \Gamma^k(z -x)$ for $z,x \in \RR^2$, $z\neq x$.  Then we define the Neumann function $\NOk$ to be the solution to
\begin{align}\label{PDE:NO}
	\left\{ 
	\begin{aligned}
		 \left( \Laplace + k^2  \right) \NOk(z,x) &= \delta_z(x) \quad &&\text{in } \RR^2\setminus \overline{\Om}\,, \\
		 \del_{\nu_x} \NOk(z,x) \MID_+\! &= 0 \quad &&\text{on } \del \Omega \,,\\
		 \Big({\frac{\partial }{\partial |x|}}-\imagi k\Big)\,\NOk(z,x)&\rightarrow 0 \quad &&\text{for } |x|\rightarrow \infty\,.
	\end{aligned}
	\right.
\end{align}
In contrast with $u_j$, $\NOk$ is not only defined on the upper half. We recall that $\NOk(z,x) = \NOk(x,z)$, which we can readily see using a Green's identity. We can express $\NOk$ as a sum of $\Gamma^k$ and a smooth remainder, which satisfies PDE (\ref{PDE:NO}) with a vanishing right-hand side in the first equation. The same holds true for $u_j$.

Using Green's identity on the convolution of $u_j$ with $\left( \Laplace + k^2  \right) \NOk$, we can infer for $i\neq j$ that
\begin{align*}
	\tfrac{1}{2} u_j(z_i) =& \int_{\RR^2_+\setminus \overline{\Om}} u_j(x)\left( \Laplace + k^2  \right) \NOk(z_i,x)\intd x \,,\\
		=&\,  I_j \, \NOk(z_j, z_i) -  \int_{\Lambda} u_j(x)\,\del_{x_2} \NOk(z_i,x) \intd\sigma_x \,.
\end{align*}
Using the trapezoidal rule we can approximate the integral in the last equation up to an error in $\OO(N^{-1})$. Here we note that $u_j$ and $\NOk$ have a logarithmic singularity at $z_j$, hence $u_j(z_j)$ is not well defined. Thus we use a slight modification in the trapezoidal rule, which is elaborated in Appendix \ref{app:A}. After such modification, we define the complex column vector $\mathbf{u}_j=(u_j(z_i))_{i=1}^N$, the complex column vector $\mathbf{N}_j=(\NOk(z_j,z_i))_{i=1}^N$ and the $N\times N$ complex matrix $\mathbf{S}=(\del_{x_2} \NOk(z_i,z_k))_{\substack{i=1,\ldots, N \\ k=1,\ldots, N}}$. Then we have that
\begin{align*}
	(\tfrac{1}{2}\mathbf{I}_N + \tfrac{2}{N+1}\mathbf{S})\,\mathbf{u}_j = I_j\,\mathbf{N}_j+\OO(N^{-1})\,,
\end{align*}
where $\mathbf{I}_N$ denotes the $N\times N$ identity matrix. 

Our objective is to solve a linear system of equation $\mathbf{A}\,\mathbf{x}\, = \, \mathbf{b}$ using a physical procedure, in which an electrical signal is applied at $z_j$, for every $j=1,\ldots, N$, and then is measured again at those points, for $\mathbf{A}\in \CC^{N\times N}$ and $\mathbf{x}, \mathbf{b} \in \CC^{N}$, where $\mathbf{A}$ and $\mathbf{b}$ are given. The scattered wave, originated at $z_j$, is in its Fourier space the function $u_j$.  Thus we are especially looking for a domain $\Omega$, which yields
\begin{align}\label{equ:S = N(A-I)}
	\mathbf{S} =  \tfrac{N+1}{2}\big(\mathbf{A} - \tfrac{1}{2}\mathbf{I}_N \big)\,,
\end{align}
and in searching so we keep track of the vector $\mathbf{N}_j$ with the intention of rapidly determining the intensities $I_s$ such that $\sum_{j=1}^N I_j\,\mathbf{N}_j = \mathbf{b}$. In this paper we primarily consider rapidly finding the domain $\Omega$ such that Equation (\ref{equ:S = N(A-I)}) holds.

\section{Asymptotic Formula for the Perturbation of the Neumann Function $\NOk$} \label{Sec:AsymptoticFormula}
Let $\Omega$ be a finite union of disjoint, non-touching, simply connected, bounded and open $\cC^2$-domains in $\RR^2$. Let $k\in (0,\infty)$, then for $z\not\in \overline{\Om}$, where $\overline{\Om}$ is the topological closure of the open set $\Omega$, we define the outside Neumann function $\NOk(z,x)$, for $x\not\in \overline{\Om}$, as the solution to the partial differential equation (\ref{PDE:NO}).
Let $B_r(\zeta)$ be a ball centred at $\zeta\in\RR^2$ with radius $r>0$. We define $\Omr$ as the union of some set $\Omega$ as defined above and of the ball $B_r(\zeta)$, where $\overline{B_r(\zeta)}$ does not intersect $\overline{\Omega}$. Let $\NOrk$ and $\NOk$ be the outside Neumann function to $\Omr$ and $\Om$, respectively.

\begin{theorem}\label{thm:NOrk-Asymptotics}
	For $k r$ small enough and for all $z,\,x \not\in \overline{\Omr}$ , $x\neq z$, we have that
	{\small
	\begin{align}\label{equ1:thm:NOrk-Asymptotics}
		\NOrk(z,x) =& \;\NOk(z,x)+\pi\,r^2\,\big(\mkern1mu k^2\,\NOk(z,\zeta)\,\NOk(x,\zeta)-2\,\nabla \NOk(z,\zeta)\cdot \nabla \NOk(x,\zeta)\big)\nonumber\\
		&+\OO(r^3\log(r))\,,
	\end{align}
	}where $\nabla$ denotes the gradient to the second input in $\NOk$ , '$\cdot$' denotes the dot-product  and $\OO(\cdot)$ denotes the limiting behaviour for $r\rightarrow 0$. Additionally, for $y\in \del(\Omega_r\setminus \Omega)$, we have that 
	{\small
	\begin{align}\label{equ2:thm:NOrk-Asymptotics}
		\NOrk(z,y) = 
			&\,\NOk(z,\zeta)+2\,(y-\zeta) \cdot \nabla \NOk(z,\zeta)+(y-\zeta)^\TransT\, (\nabla\nabla^\TransT) \NOk(z,\zeta)\, (y-\zeta)\nonumber\\
			&- r^2k^2\frac{1}{2}\, \NOk(z,\zeta)\,(\imagi \frac{\pi}{2}-\gamma-0.5-\log\big(\frac{k\, r}{2}\big)-2\pi\,\ROk(\zeta,\zeta))\nonumber\\
			&-r^2\, 2\pi \,\nabla_w\ROk(\zeta,w)\mid_{w=\zeta}\cdot\nabla \NOk(z,\zeta)+\OO(r^3\log(r))\,,
	\end{align}
	}where $(\nabla\nabla^\TransT)$ denotes the Hessian matrix, $\gamma\approx 0.57721$ denotes the Euler–Mascheroni constant, and  $\ROk(z,w)\DEF \NOk(z,w)-\Gamma^k(z,w)$ has a removable singularity at $w = z$.\\
	For $k\cdot r>0$ small enough and for all $y,\,w \in \del{\Omr}$ , $y\neq w$, we have that
	{\small
	\begin{align}\label{equ3:thm:NOrk-Asymptotics}
		\NOrk(y,w) = 
			&\, \frac{\log(1-\cos(\theta_y\!-\! \theta_w))}{2\pi}+\frac{2\log(kr)+2\gamma-\imagi \pi}{4\pi}+\ROk(\zeta,\zeta)+\OO(r\log(r))\,.
	\end{align}
	}
\end{theorem}

\begin{proof}
	Using Green's identity and the PDE (\ref{PDE:NO}) we readily see that
	\begin{align}\label{proof:Asymp:equ:NOrk=NOk+int}
		\NOrk(z,x) 
			&= \int_{\RR^2\setminus \Omr} \left( \Laplace + k^2  \right)\NOk(x,y)\;\NOrk(z,y)\intd y\nonumber\\
			&= \NOk(x,z) - \int_{\del B_r(\zeta)}\del_{\nu_y} \NOk(x,y) \;\NOrk(z,y)\intd \sigma_y\,,
	\end{align}
	where the normal vector still points outwards. Using an analogous argument by integrating $\NOk$ with itself, we see that $\NOk(x,z)=\NOk(z,x)$. We let $x$ go to $\del B_r(\zeta)$ and apply the gradient on both sides and apply the normal at $y$. Then we obtain the equation
	\begin{align}\label{proof:Asymp:equ:-dN=-d_Int_dN N}
		-\del_{\nu_y} \NOk(z,y) = - \del_{\nu_y} \int_{\del B_r(\zeta)}\del_{\nu_w} \NOk(y,w) \;\NOrk(z,w)\intd \sigma_w\,,
	\end{align}
	where $y$ as well as $w$ are elements of $\del B_r(\zeta)$. 
	We remark here that we cannot pull the normal derivative inside the integral. 
	Let us consider next the decomposition $\NOk(y,w)=\Gk(y,w)+\ROk(y,w)$, where $\Gk(y,w)$ is the fundamental solution to the Helmholtz equation, that means that $( \Laplace_w + k^2  ) \Gk(y,w) = \delta_y(w)$, and $\ROk(y,w)$ is the remaining part of the PDE (\ref{PDE:NO}). $\Gk$ can be expressed through $\Gk(y,w)=-\frac{\imagi}{4} H_0^{(1)}(k|y-w|)$, where $H_0^{(1)}$ is the Hankel function of first kind of order zero and $\ROk$ is smooth \cite{OutSideNeumannExistence}.  
	From the decomposition in Equation (\ref{proof:Asymp:equ:-dN=-d_Int_dN N}) to arrive at
	\begin{align*}
		\del_{\nu_y} \NOk(z,y) = \del_{\nu_y} \int_{\del B_r(\zeta)}\del_{\nu_w} \Gk(y,w) \;\NOrk(z,w)\intd \sigma_w + \OO(r)\,,
	\end{align*}
	by using the fact that the integral over $\del B_r(\zeta)$ decays linearly for $r\rightarrow 0$.  Transforming the normal derivative in the integral using polar coordinates, where we use that we have an integral over the boundary of a circle, we can infer that
	{\footnotesize
	\begin{align*}
	\del_{\nu_y} \NOk(z,y(\tau)) &
		= \lim_{h\searrow 0} 
		\int_{0}^{2\pi}
			\frac{\imagi k\,r}{4}
			\Big[
			\frac{k(h+r(1-\cos(\Delta))(-h\cos(\Delta)+r(1-\cos(\Delta)))H_0^{(1)}(k|\cdot|)}{|\cdot|^2}\nonumber\\
		 		+
		 		&\frac{(-2r(h+r)+(h^2+2hr+2r^2)\cos(\Delta))H_1^{(1)}(k|\cdot|)}{|\cdot|^3}
		 	\Big] \,
		 	\NOrk(z,w(t))
		\,\intd t
		+ \OO(r)\,,
	\end{align*}
	}
	where $\Delta\DEF t-\tau$, where $y(\tau)=\zeta+r\compactvec{\cos(\tau)\\ \sin(\tau)}$ and $w(t)=\zeta+r\compactvec{\cos(t)\\ \sin(t)}$ and $|\cdot|\DEF\sqrt{h^2+(1-\cos(\Delta))(2h r +2 r^2)}$. 
	
	With the Taylor series for the Hankel function $H_0^{(1)}$ and Hankel function $H_1^{(1)}$, for $k\, r$ small enough, and considering the asymptotic behaviour of the forthcoming terms and applying some trigonometric identities we readily see that
	{\footnotesize
	\begin{align*}
	\del_{\nu_y}& \NOk(z,y(\tau)) 
		=\lim_{h\searrow 0} 
		\frac{r}{2\pi}
		\int_{0}^{2\pi}
			\frac{h^2 -2\sin(\frac{\Delta}{2})^2\,(h^2\!+\!2hr\!+\!2r^2)}{(h^2+2\sin(\frac{\Delta}{2})^2\,(2hr\!+\!2r^2))^2}
		 	\NOrk(z,w(t))
		 	\intd t
		 	+
		 	\OO(r\,\log(r))\,.
	\end{align*}
	}
	Using integration by parts, where we consider that $\NOrk(z,w(\cdot))$ is a periodic function, we obtain that
	\begin{align*}
	\del_{\nu_y}& \NOk(z,y(\tau)) 
		=\lim_{h\searrow 0} 
		\frac{-1}{4\pi\, r}
		\int_{0}^{2\pi}
			\frac{\sin(t-\tau)\,\del_t \NOrk(z,y(t))}{2\sin\big(\frac{t-\tau}{2}\big)^2+h}
		\,\intd t
		+\OO(r\,\log(r))\,.
	\end{align*}
	Before we can proceed, we have to study a linear operator we call $\Hcirc$, which takes a $2\pi$ periodic $\cC^2$ function $\phi$ and maps it to 
	\begin{align*}
		\Hcirc[\phi](\tau)
		\DEF
			\frac{1}{2\pi}\lim_{h\searrow 0} 
			\int_{0}^{2\pi}
				\frac{\sin(t-\tau)}{2\sin\big(\frac{t-\tau}{2}\big)^2+h}
		 		\,\phi(t)
			\,\intd t.
	\end{align*}
	We can readily show that 
	$$
	\Hcirc[\phi](\tau)=\frac{1}{2\pi}\pvint_{0}^{2\pi}\phi(t)\cot\Big(\frac{t-\tau}{2}\Big)\intd t\,,
	$$
	where 'p.v.' stands for the 'principle value'. This equation follows by integration by parts on both sides of the equation, and by using the integrability of the logarithm function. Now, we can state that $\Hcirc$ is an invertible operator up to a constraint, according to \cite[§ 28]{Muskhelishvili}, and that the solution to $\Hcirc[\phi]=\psi$ is given through $\phi=-\Hcirc[\psi]$, where the constraint is that $\int_0^{2\pi}\phi=0$. 
	Then we can infer that 
	\begin{align*}
		\Hcirc[\del_{\nu_y} \NOk(z,y(\cdot))](t)
		=	\frac{1}{2\,r}\,\del_t \NOrk(z,y(t))+ \OO(r\log(r))\,,
	\end{align*}
	where we used that $\Hcirc[\OO(r\log(r))]=\OO(r\log(r))$.
	Thus it follows that
	\begin{align*}
		\NOrk(z,y(t))
		=	C+2\,r\,\int\Hcirc[\del_{\nu_y} \NOk(z,y(\cdot))]
			+ \OO(r^2\log(r))\, ,
	\end{align*}
	for a constant function $C$ in $t$.
	Next, we approximate the known function $\NOk$ through
	\begin{align*}
		\del_{\nu_y} \NOk(z,y(t)) 
			= \compactvec{\cos(t)\\ \sin(t)} \cdot \nabla_y \NOk(z,y)|_{y=\zeta}+\OO(r)\,.
	\end{align*}
	Using that $\Hcirc[\sin(\cdot)]=\cos(\cdot)$ , we see that
	\begin{align}\label{proof:Asymp:equ:NOrk=C+2rdN+O}
		\NOrk(z,y(t))
		=	C+2\, r\,\compactvec{\cos(t)\\ \sin(t)} \cdot \nabla \NOk(z,\zeta)
			+ \OO(r^2\log(r))\,.
	\end{align}
	Analogously to Equation (\ref{proof:Asymp:equ:NOrk=NOk+int}), we can formulate the statement that
	\begin{align}\label{proof:Asymp:equ:NOrk=NOk+2int}
		\NOrk(z,y) 
			&= 2\,\NOk(z,y) - 2\int_{\del B_r(\zeta)}\del_{\nu_w} \NOk(y,w) \;\NOrk(z,w)\intd \sigma_w\,,
	\end{align}
	where $z\not\in \overline{\Omega}$ and $y\in \del B_r(\zeta)$, and where we use that the Dirac measure located at $y$, which is at the boundary of the integration domain, which is a $\cC^2$ boundary, yields only half of the evaluation of the integrand at $y$.
	 We then apply Equation (\ref{proof:Asymp:equ:NOrk=C+2rdN+O}) to the last equation and see that $C=\NOk(z,\zeta)$. Applying it again for the second order term, while using Taylor expansions and comparing coefficients of the same order in $r$, we readily obtain the second equation in Theorem \ref{thm:NOrk-Asymptotics}. For the first equation in Theorem \ref{thm:NOrk-Asymptotics}, we apply the formula for $\NOrk(z,y(t))$, and the Taylor expansion up to second order for $\del_{\nu_y} \NOk(z,y(t))$ to Equation (\ref{proof:Asymp:equ:NOrk=NOk+int}) to obtain
	 {\small
	 \begin{align*}
	 	&\NOrk(z,x) = \NOk(z,x)- 2\,r\int_0^{2\pi}r\,\del_\nu \NOk(x,\zeta)\,\compactvec{\cos(t)\\ \sin(t)}\cdot\nabla\NOk(z,\zeta)\intd t+\OO(r^3\log(r))\\
	 				&- \NOk(z,\zeta)\int_0^{2\pi}r\,
	 					\Big( 
	 						\compactvec{\cos(t)\\ \sin(t)}\cdot \nabla \NOk(x,\zeta)
	 						+ r \,\compactvec{\cos(t)\\ \sin(t)}^\TransT \, (\nabla\nabla^\TransT) \NOk(x,\zeta)\,\compactvec{\cos(t)\\ \sin(t)}
	 					\Big)\intd t\,,
	 \end{align*}
	 }where $(\nabla\nabla^\TransT)$ denotes the Hessian matrix which emerges from the Taylor expansion. We evaluate the two integrals explicitly, use that $\Laplace \NOk = -k^2\NOk $ and obtain Equation (\ref{equ1:thm:NOrk-Asymptotics}).
	 For Equation (\ref{equ3:thm:NOrk-Asymptotics}) we use Green's identity and obtain
	 {\small
	 \begin{align*}
	 	\NOrk(y,w) &= 2\NOk(y,w)-2 \int_{\del\Om_r}\del_{\nu_u} \NOk(w,u)\NOrk(y,u)\intd \sigma_u\\
	 	&= \tfrac{1}{2\pi}\log(1-\cos(\theta_y-\theta_w))+\frac{\log(\tfrac{kr}{\sqrt{2}})}{\pi}+\frac{2\gamma-\imagi \pi}{2\pi}+2\ROk(\zeta,\zeta)\\
	 	&\;\;\; -2r\int_{-\pi}^\pi\frac{1}{2\pi\, 2 r}\,\NOrk(y,u(t))\intd t+\OO(r\,\NOrk(y,u))\,.
	 \end{align*}
	 }
	 Solving for $\frac{1}{2\pi}\int_{-\pi}^\pi\,\NOrk(y,u(t))\intd t$ and substituting we obtain Equation (\ref{equ3:thm:NOrk-Asymptotics}).
\end{proof}


	Let $R>r>0$ and let $\OmR$ be defined in the way that $\Omr$ was introduced, that is $\OmR$ is a ball of radius $R$ at $\zeta\in \RR^2$ adjoined to the domain $\Om$, hence $\Omr\subsetneq\OmR$. Then for any $z_r\in \del B_r(\zeta)$, we define $z_R\in \del B_R(\zeta)$ to be the projection of $z_r$ along the normal vector to $B_R(\zeta)$. Thus we have that $z_R(t_z)= R\,(\cos(t_z),\sin(t_z))^\TransT+\zeta$ , for $t_z\in (-\pi,\pi)$. 
	
\begin{lemma}\label{lemma:Sings}
	Let $R>r>0$, for all $z_r,\,x_r \in \del{\Omr}$, $z_R,\,x_R \in \del{\Omr}$ we have that
	{\small
	\begin{align}
		\NOrk(z_r(t_z),x_r(t_x)) 
				&=  \frac{1}{2\pi}\log(1\!-\!\cos(t_z\!-\!t_x)) + Q_1(t_z, t_x)\,,\label{equ:lemma:1}\\
		\NOrk(z_R(t_z),x_r(t_x)) 
				&=  \frac{1}{2\pi}\log(1\!-\!\cos(t_z\!-\!t_x)) + Q_1(t_z, t_x)\,,\label{equ:lemma:12}\\
		\NOrk(z_R(t_z),x_R(t_x)) 
				&=  \frac{1}{4\pi}\log(1\!-\!\cos(t_z\!-\!t_x)) \nonumber\\
					+ &\frac{1}{4\pi}\log(R^4\!+\!r^4\!-\!2\,R^2r^2\cos(t_z\!-\!t_x)) 
					+ Q_2(t_z, t_x)\,,\label{equ:lemma:2}\\
		\del_{\nu_{x_R}}\NOrk(z_R(t_z),x_R(t_x)) 
				&=  \frac{-1}{4\pi R}  
					+ \frac{1}{2\pi R} \frac{r^2\,(R^2\cos(t_z\!-\!t_x)-r^2)}{R^4+r^4-2\,R^2 r^2 \cos(t_z\!-\! t_x)} 
					+ Q_3(t_z, t_x)\,,\label{equ:lemma:3}
	\end{align}
	}
	where $Q_1, Q_2, Q_3$ have removable singularities at $t_x = t_z$, when $R = r$.
\end{lemma}
	
\begin{proof}
	Equations (\ref{equ:lemma:1}) and (\ref{equ:lemma:12}) follow by readily using Green's identity on the convolution of $\NOrk$ with $\Gamma^k$, and PDE (\ref{PDE:NO}), where we have to consider that an integral whose integration-boundary is over the singularity of the Dirac measure leads to half of the evaluation of the integrand. 
	
	For Equation (\ref{equ:lemma:2}), its proof is a simplification of the derivation of Equation (\ref{equ:lemma:3}). For Equation (\ref{equ:lemma:3}) we have with Green's identity that
	\begin{align*}
		\del_{\nu_{x_R}}\!\NOrk(z_R,x_R) =
			\del_{\nu_{x_R}}\!\NOk(z_R,x_R) \!- \!\int_{\del B_r(\zeta)} \!\!\del_{\nu_{x_R}}\del_{\nu_{y_r}}\NOk(x_R,y_r) \NOrk(z_R,y_r)\intd \sigma_{y_r}\,.
	\end{align*}
	Splitting $\NOk$ in its singular part $\Gamma^k$ and its smooth remainder and subsequently extracting the singularity in $\Gamma^k$, and doing so for $\NOrk$ as well, where we use Equation (\ref{equ:lemma:12}), we obtain that 
	\begin{align*}
		&\del_{\nu_{x_R}}\NOrk(z_R,x_R) =
			\frac{1}{2\pi R} \frac{r\,(R\cos(t_z\!-\!t_x)-r)}{R^2+r^2-2 R r \cos(t_z\!-\!t_x)} + Q_3(t_z, t_x)\\
			&\;- r\int_{-\pi}^\pi   
				\!\!\frac{-1}{2 \pi}\frac{2 R r-(R^2+r^2)\cos(t_x-t)}{(R^2+r^2-2 R r \cos(t_x-t))^2} \, 
				\!\frac{1}{2 \pi}\!\log\Big(\frac{R^2+r^2}{2\,R\,r}\!-\!\cos(t_z-t)\Big)\intd t\,.
	\end{align*}
	Using the technical derivation shown in Appendix \ref{app:A} we prove Equation (\ref{equ:lemma:3}).
\end{proof}	
	
We decompose $\NOrk(z_r(t_z),x_r(t_x))$, for $z_r, x_r\in \del B_r(\zeta)$, into its singular part and a smooth enough part, that is,
\begin{align*}
	\NOrk(z_r(t_z),x_r(t_x))
		&= \frac{1}{2\pi}\log(1-\cos(t_z-t_x)) + \widetilde{\NOrk}(z_r(t_z),x_r(t_x))\,,
\end{align*}
and furthermore we express $\widetilde{\NOrk}$ through a Fourier series as
\begin{align}\label{SmoothFourierExp}
	\widetilde{\NOrk}(z_r(t_z),x_r(t_x)) = \sum_{n=0}^\infty p_{z_r}^{(n)} \cos(n\, t_x)+ q_{z_r}^{(n)} \sin(n\, t_x)\,.
\end{align}

\begin{theorem}\label{thm:RmrApprox}
	For $k r>0$ small enough and for all $z,\,x \not\in \overline{\Omr}$ , $x\neq z$, we have that
	{\small
	\begin{align}\label{equ1:thm:RmrApprox}
		\NOrk(z_R,x_r) = &\,\NOrk(z_r,x_r) +\frac{1}{2\pi}\log\Big(\frac{\frac{R^2+r^2}{2\,R\,r}-\cos(t_z-t_x)}{1-\cos(t_z-t_x)}\Big)+\OO_{L^2}\Big(\frac{(R-r)^2}{r^2}\Big)\,,
	\end{align}
	}
	where the $\OO_{L^2}$ term is a function with a $L^2(\del\Om_r)$ norm, which is in $\OO\big(\frac{(R-r)^2}{r^2}\big)$, in the $x_r$ variable. Moreover, 
	{\small
	\begin{align}\label{equ2:thm:RmrApprox}
		\del_{\nu_{x_R}}\NOrk(z_R,x_R) 
			&= \del_{\nu_{x_R}}\NOk(z_R,x_R) +\frac{r^2}{2\pi\, R}\frac{R^2\cos(t_z-t_x)-r^2}{R^4+r^4-2 R^2 r^2\cos(t_z-t_x)}\nonumber\\
			&\mkern-110mu- r\int_{-\pi}^\pi \del_{\nu_{x_R}}\del_{\nu_{y_r(t)}} \widetilde{\NOk}(x_R,y_r(t)) \, 
			\Big(
				\tfrac{1}{2\pi}\log\big(\tfrac{R^2+r^2}{2\,R\,r}-\cos(t_z\!-\! t)\big)
				+\widetilde{\NOrk}(z_r,y_r(t))
			\Big)\intd t\nonumber\\
			&\mkern-110mu -\frac{1}{2 R}\sum_{n=1}^\infty n\,(\tfrac{r}{R})^n\big(p_{z_r}^{(n)} \cos(n\, t_x)+ q_{z_r}^{(n)} \sin(n\, t_x)\big)\nonumber\\
			&\mkern-110mu +\,\OO\Big(\frac{(R-r)^2}{r}\Big)\,,
	\end{align}
	}
	where
	{\small
	\begin{align*}
		\del_{\nu_{x_R}}\del_{\nu_{y_r}}\widetilde{\NOk}(x_R,y_r)
			\DEF \del_{\nu_{x_R}}\del_{\nu_{y_r}}\NOk(x_R,y_r)-\frac{-1}{2 \pi}\frac{2 R r-(R^2+r^2)\cos(t_x-t_y)}{(R^2+r^2-2 R r \cos(t_x-t_y))^2}\,.
	\end{align*}
	}
	Furthermore,  we have
	{\small 
	\begin{align}\label{equ3:thm:RmrApprox}
		 \NORk(z_R,x_R) 
			=& \,\NOrk(z_r,x_r)+(\NOk(z_R,x_R)-\NOk(z_r,x_r))\nonumber\\
			&\mkern-90mu -r\, (R\!-\! r)\int_{-\pi}^\pi \del_{\nu_{x_r}}\del_{\nu_{y_r(t)}} \widetilde{\NOk}(x_r,y_r(t))\NOrk(z_r,y_r(t))\intd t\nonumber\\
			&\mkern-90mu -r\int_{-\pi}^\pi \del_{\nu_{y_r(t)}} \widetilde{\NOk}(x_r,y_r(t)) \, 
				\tfrac{1}{2\pi}\log\bigg(\frac{\tfrac{R^2+r^2}{2\,R\,r}-\cos(t_z-t)}{1-\cos(t_z-t)}\bigg)
			\intd t\nonumber\\
			&\mkern-90mu  -\tfrac{1}{2}\sum_{n=1}^\infty\big((\tfrac{r}{R})^n-(\tfrac{r}{R})^{2n}\big)\big(p_{z_r}^{(n)} \cos(n\, t_x)+ q_{z_r}^{(n)} \sin(n\, t_x)\big)\nonumber\\
			&\mkern-90mu - R\int_{-\pi}^\pi  \widetilde{\del_{\nu_{y_R(t)}}\NOrk}(x_R,y_R(t))  \NOrk(z_r,y_r(t))\intd t\nonumber\\
			&\mkern-90mu +\,\OO\Big(\frac{(R-r)^2}{r}\Big)\,,
	\end{align}
	}
	where
	{\small
	\begin{align*}
		\del_{\nu_{y_r}} \widetilde{\NOk}(x_R,y_r)
			&\DEF \del_{\nu_{y_r}} \NOk(x_R,y_r)-\frac{1}{2\pi}\frac{r-R\cos(t_x\!-\! t_y)}{R^2+r^2-2Rr\cos(t_x\!-\! t_y)}\,, \\
		\widetilde{\del_{\nu_{x_R}}\NOrk}(z_R,x_R)
			&\DEF \del_{\nu_{x_R}}\NOrk(z_R,x_R) -\frac{1}{2\pi\,2 R}-\frac{r^2}{2\pi\, R}\frac{R^2\cos(t_z-t_x)-r^2}{R^4+r^4-2 R^2 r^2\cos(t_z-t_x)} \,.
	\end{align*}
	}
\end{theorem}

The idea of proving this theorem is to extract the singularities developed in Lemma \ref{lemma:Sings} in the integral expression for $\NORk$.  Then any explicitly appearing integrals are solved in a similar way as described in Appendix \ref{app:A} by using Fourier series.

\begin{proof}
	Assuming $z_r\neq x_r$, we can use Taylor's theorem to obtain that 
	{\small 
	\begin{align*}
		\NOrk(z_R,x_r) = \NOrk(z_r,x_r)+(R-r)\del_{\nu_{z_r}}\NOrk(z_r,x_r)+\tfrac{1}{2}(R-r)^2\del_{\nu_{z_r}}^2\NOrk(w_{R,r},x_r)\,,
	\end{align*}}
	for some $w_{R,r}\in \RR^2$ between $z_R$ and $z_r$. We note that $\del_{\nu_{z_r}}\NOrk(z_r,x_r)=0$. We need the term $\tfrac{1}{2}(R-r)^2\del_{\nu_{z_r}}^2\NOrk(w_{R,r},x_r)$ to be in $\OO_{L^2}$, but that is not the case due to the singular term in $\NOrk$. Hence we extract the singular term from $\NOrk(z_R,x_r)$ and then we can infer that
	{\small \begin{align*}
		\NOrk(z_R,x_r) 
			&= \tfrac{1}{2\pi}\log(R^2+r^2-2Rr\cos(t_z-t_x))+\widetilde{\NOrk}(z_R,x_r)\\
			&= \tfrac{1}{2\pi}\log(R^2+r^2-2Rr\cos(t_z-t_x))+\widetilde{\NOrk}(z_r,x_r)\\
				&\mkern+50mu +(R-r)\del_{\nu_{z_r}}\widetilde{\NOrk}(z_r,x_r)+\tfrac{1}{2}(R-r)^2\del_{\nu_{z_r}}^2\widetilde{\NOrk}(w_{R,r},x_r)\\
			&= \tfrac{1}{2\pi}\log\Big(\frac{R^2+r^2-2Rr\cos(t_z-t_x)}{2 r^2-2 r^2 \cos(t_z-t_x)}\Big)+\NOrk(z_r,x_r)\\
				&\mkern+50mu +(R-r)\big(0-\tfrac{1}{2\pi\, r}\big)+\OO((R-r)^2)\,.
	\end{align*}}
	Extracting the term $\log(\tfrac{R}{r})$ from the logarithm term and using the Taylor approximation for $R\rightarrow r$, on that extraction, we then obtain Equation (\ref{equ1:thm:RmrApprox}). 
	Considering Green's identity we have that
	{\small \begin{align*}
		\NOrk(z_R,x_R)=\NOk(z_R,x_R) - r \int_{-\pi}^\pi \del_{\nu_{y_r(t)}} \NOk(x_R,y_r(t)) \NOrk(z_R,y_r(t))\intd t\,.
	\end{align*}}
	Next we apply $\del_{x_R}$ on both sides and then interchange the integral and $\del_{x_R}$. This leads to the term $\del_{\nu_{x_R}}\del_{\nu_{y_r}} \NOk(x_R,y_r)$, whose singular part we extract from $\del_{\nu_{x_R}}\del_{\nu_{y_r}} \NOk(x_R,y_r)$. The equation then reads
	{\small \begin{align*}
		\del_{\nu_{x_R}}\NOrk&(z_R,x_R) =\del_{\nu_{x_R}}\NOk(z_R,x_R) - r\!\int_{-\pi}^\pi \del_{\nu_{x_R}}\del_{\nu_{y_r(t)}} \widetilde{\NOk}(x_R,y_r(t)) \NOrk(z_R,y_r(t))\intd t\\
			&- r \int_{-\pi}^\pi  \frac{-1}{2 \pi}\frac{2 R r-(R^2+r^2)\cos(t_x-t)}{(R^2+r^2-2 R r \cos(t_x-t))^2} \NOrk(z_R,y_r(t))\intd t\,.
	\end{align*}}
	Then we use Equation (\ref{equ1:thm:RmrApprox}) and this leads us to the equation
	 {\small \begin{align}\label{equ4:thm:RmrApprox}
		\del_{\nu_{x_R}}\NOrk & (z_R,x_R) =\del_{\nu_{x_R}}\NOk(z_R,x_R) \nonumber\\
			&- r\int_{-\pi}^\pi \del_{\nu_{x_R}}\del_{\nu_{y_r(t)}} \widetilde{\NOk}(x_R,y_r(t)) \, \tfrac{1}{2 \pi}\log\big(\tfrac{R^2+r^2}{2\,R\,r}-\cos(t_z-t)\big)\intd t\nonumber\\
			&- r\int_{-\pi}^\pi \del_{\nu_{x_R}}\del_{\nu_{y_r(t)}} \widetilde{\NOk}(x_R,y_r(t)) \, \widetilde{\NOrk}(z_r,y_r(t))\intd t\nonumber\\
			&- r \int_{-\pi}^\pi  \frac{-1}{2 \pi}\frac{2 R r-(R^2+r^2)\cos(t_x-t)}{(R^2+r^2-2 R r \cos(t_x-t))^2} \, \tfrac{1}{2 \pi}\log\big(\tfrac{R^2+r^2}{2\,R\,r}-\cos(t_z-t)\big)\intd t\nonumber\\
			&- r \int_{-\pi}^\pi  \frac{-1}{2 \pi}\frac{2 R r-(R^2+r^2)\cos(t_x-t)}{(R^2+r^2-2 R r \cos(t_x-t))^2} \, \widetilde{\NOrk}(z_r,y_r(t))\intd t \nonumber\\
			&- r \int_{-\pi}^\pi  \frac{-1}{2 \pi}\frac{2 R r-(R^2+r^2)\cos(t_x-t)}{(R^2+r^2-2 R r \cos(t_x-t))^2} \, (\tfrac{1}{2 \pi}\log(\nicefrac{r}{R})-\tfrac{R-r}{2\pi r}+\OO_{L^2}((R-r)^2))\intd t\nonumber\\
			&+\OO\Big(\frac{(R-r)^2}{r}\Big)
			\,.
	\end{align}}
	Note that
	{\small 
	\begin{align}\label{equ5:thm:RmrApprox}
		\int_{-\pi}^\pi  \frac{-1}{2 \pi}\frac{2 R r-(R^2+r^2)\cos(t_x-t)}{(R^2+r^2-2 R r \cos(t_x-t))^2} \cos(n (t_z-t))\intd t=\frac{n}{2 R r}\Big(\frac{r}{R}\Big)^n\cos(n(t_z-t_x))\,,
	\end{align}
	}
	for all $n\in\NN_0$, which we can readily show from the $2\pi$-periodicity by using trigonometric formulas and applying an induction on $n\geq 1$. Furthermore, we have that
	{\small \begin{align*}
		\frac{1}{2 \pi}\log\big(\tfrac{R^2+r^2}{2\,R\,r}-\cos(t_z-t)\big)
			= \frac{1}{2 \pi}\log\big(\tfrac{R}{2\,r}\big)
			- \frac{1}{\pi}\sum_{n=1}^\infty\frac{(\nicefrac{r}{R})^n}{n}\cos(n(t_z-t))\,,
	\end{align*}
	}
	With that identity we can determine all integrals in Equation (\ref{equ4:thm:RmrApprox}) and show Equation (\ref{equ2:thm:RmrApprox}). For an elaborated calculation of the third integral, see Appendix \ref{app:A}.\\
	Using Green's identity on $\NORk(z_R,x_R)$, we can infer that
	{\small 
	\begin{align}\label{equ6:thm:RmrApprox}
		\NORk(z_R,x_R) = 2\NOrk(z_R,x_R) 
			- 2 R\int_{-\pi}^\pi  \del_{\nu_{y_R(t)}}\NOrk(x_R,y_R(t))  \NORk(z_R,y_R(t))\intd t\,.
	\end{align}
	}
	Similar to the derivation of Equation (\ref{equ2:thm:RmrApprox}), we can compute that
	{\small 
	\begin{align*}
		\NOrk&(z_R,x_R) = \NOk(z_R,x_R) \nonumber\\
			&- r\int_{-\pi}^\pi \del_{\nu_{y_r(t)}} \widetilde{\NOk}(x_R,y_r(t)) \, \tfrac{1}{2 \pi}\log\big(\tfrac{R^2+r^2}{2\,R\,r}-\cos(t_z-t)\big)\intd t\nonumber\\
			&- r\int_{-\pi}^\pi \del_{\nu_{y_r(t)}} \widetilde{\NOk}(x_R,y_r(t)) \, \widetilde{\NOrk}(z_r,y_r(t))\intd t\nonumber\\
			&- r \int_{-\pi}^\pi  \frac{1}{2 \pi}\frac{r-R\cos(t_x-t)}{R^2+r^2-2 R r \cos(t_x-t)} \, \tfrac{1}{2 \pi}\log\big(\tfrac{R^2+r^2}{2\,R\,r}-\cos(t_z-t)\big)\intd t\nonumber\\
			&- r \int_{-\pi}^\pi  \frac{1}{2 \pi}\frac{r-R\cos(t_x-t)}{R^2+r^2-2 R r \cos(t_x-t)} \, \widetilde{\NOrk}(z_r,y_r(t))\intd t \nonumber\\
			&- r \int_{-\pi}^\pi  \frac{1}{2 \pi}\frac{r-R\cos(t_x-t)}{R^2+r^2-2 R r \cos(t_x-t)} \, (\tfrac{1}{2 \pi}\log(\nicefrac{r}{R})-\tfrac{R-r}{2\pi r}+\OO_{L^2}((R-r)^2))\intd t\nonumber\\
			&+\OO\Big(\frac{(R-r)^2}{r}\Big)
			\,.
	\end{align*}
	}
	Using that
	{\small 
	\begin{align}\label{equ7:thm:RmrApprox}
		\int_{-\pi}^\pi  \frac{1}{2 \pi}\frac{r-R\cos(t_x\!-\! t)}{R^2+r^2-2 R r \cos(t_x\!-\! t)} \cos(n (t_z\!-\! t))\intd t
			=\begin{cases}
    			0\,,& \text{if } n=0\,,\\
    			\frac{-1}{2\pi\,r}(\tfrac{r}{R})^n\cos(n(t_x\!-\! t_z)),              & \text{if } n\geq 1\,,
			\end{cases}
	\end{align}
	}
	we readily see that
	{\small 
	\begin{align}\label{equ8:thm:RmrApprox}
		\NOrk(z_R,x_R) &= \NOk(z_R,x_R) +\frac{1}{4\pi}\log\Big(\frac{R^4+r^4-2 R^2 r^2\cos(t_z\!-\!t_x)}{R^4}\Big)\\
			&\mkern-70mu- r\int_{-\pi}^\pi \del_{\nu_{y_r(t)}} \widetilde{\NOk}(x_R,y_r(t)) \, 
			\Big(
				\tfrac{1}{2\pi}\log\big(\tfrac{R^2+r^2}{2\,R\,r}-\cos(t_z-t)\big)
				+\widetilde{\NOrk}(z_r(t_z),y_r(t))
			\Big)\intd t\nonumber\\
			&\mkern-70mu +\tfrac{1}{2}\widetilde{\NOrk}(z_r,x_r) -\tfrac{1}{2} p_{z_r}^{(0)}-\tfrac{1}{2}\sum_{n=1}^\infty\big(1- (\tfrac{r}{R})^n\big)\big(p_{z_r}^{(n)} \cos(n\, t_x)+ q_{z_r}^{(n)} \sin(n\, t_x)\big)\nonumber\\
			&\mkern-70mu +\,\OO\Big(\frac{(R-r)^2}{r}\Big)\,,
	\end{align}
	}
	where the logarithm term is derived similarly as in Appendix \ref{app:A}. We consider the integral term in Equation (\ref{equ6:thm:RmrApprox}). To this end we will apply Equation (\ref{equ2:thm:RmrApprox}) and consider the singular parts of $\del_{\nu_{x_R}}\NOrk(z_R,x_R)$. Hence we define
	{\small 
	\begin{align*}
		\del_{\nu_{x_R}}\NOrk(z_R,x_R) = \frac{1}{2\pi\,2 R}+\frac{r^2}{2\pi\, R}\frac{R^2\cos(t_z-t_x)-r^2}{R^4+r^4-2 R^2 r^2\cos(t_z-t_x)}+\widetilde{\del_{\nu_{x_R}}\NOrk}(z_R,x_R) \,.
	\end{align*}
	}
	Consider that $\widetilde{\del_{\nu_{x_R}}\NOrk}(z_R,x_R)$ is of order $\OO(R-r)$, because using Taylor series we have that
	{\small 
	\begin{align*}
		\widetilde{\del_{\nu_{x_R}}\NOrk}(z_R,x_R) 
			=&\, \widetilde{\del_{\nu_{x_R}}\NOrk}(z_R,x_r)+\OO(R-r)\\
			=&\, \del_{\nu_{x_R}}\NOrk(z_R,x_r)+\OO(R-r)\\
			 &\mkern-100mu- \Big(\frac{r-R\cos(t_z\!-\!t_x)}{2\pi(R^2+r^2-2 R r \cos(t_z\!-\!t_x))}+\frac{1}{2\pi \,r}\frac{r^3(R\cos(t_z\!-\!t_x)-r)}{r^2(R^2+r^2-2 R r \cos(t_z\!-\!t_x))}\Big)\\
			=&\, \OO(R-r)\,.
	\end{align*}
	}
	Then, applying the singular decomposition to the integral in Equation (\ref{equ6:thm:RmrApprox}), and using the same techniques as are those used in Appendix \ref{app:A}, we have that 
	{\small 
	\begin{align*}
		\NORk(z_R,x_R) 
			=& \,2\,\NOrk(z_R,x_R) +\frac{\log(2)}{4 \pi}-\frac{1}{2\pi}\int_{-\pi}^{\pi}\widetilde{\NORk}(z_R,y_R(t))\intd t\\
			&\mkern-90mu -\frac{1}{2\pi}\log\bigg(\frac{R^4+r^4-2R^2r^2\cos(t_z\!-\!t_x)}{R^4(1-\cos(t_z\!-\!t_x))}\bigg)-\NORk(z_R,x_R)\\
			&\mkern-90mu +\tfrac{1}{2\pi}\int_{-\pi}^{\pi}\widetilde{\NORk}(z_R,y_R(t))\intd t+\sum_{n=1}^\infty\big(1- (\tfrac{r}{R})^{2n}\big)\big(p_{z_R}^{(n)} \cos(n\, t_x)+ q_{z_R}^{(n)} \sin(n\, t_x)\big)\\
			&\mkern-90mu - 2 R\int_{-\pi}^\pi  \widetilde{\del_{\nu_{y_R(t)}}\NOrk}(x_R,y_R(t))  \NORk(z_R,y_R(t))\intd t\,.
	\end{align*}
	}
	Then we can apply Equation (\ref{equ8:thm:RmrApprox}) and obtain
	{\small 
	\begin{align}\label{equ9:thm:RmrApprox}
		2 \NORk(z_R,x_R) 
			=& \,2\,\NOk(z_R,x_R)+\NOrk(z_r,x_r)+\frac{\log(2)}{2 \pi}- p_{z_r}^{(0)}\nonumber\\
			&\mkern-90mu- 2r\int_{-\pi}^\pi \del_{\nu_{y_r(t)}} \widetilde{\NOk}(x_R,y_r(t)) \, 
			\Big(
				\tfrac{1}{2\pi}\log\big(\tfrac{R^2+r^2}{2\,R\,r}-\cos(t_z-t)\big)
				+\widetilde{\NOrk}(z_r,y_r(t))
			\Big)\intd t\nonumber\\
			&\mkern-90mu  -\sum_{n=1}^\infty\big(1- (\tfrac{r}{R})^n\big)\big(p_{z_r}^{(n)} \cos(n\, t_x)+ q_{z_r}^{(n)} \sin(n\, t_x)\big)\nonumber\\
			&\mkern-90mu +\sum_{n=1}^\infty\big(1- (\tfrac{r}{R})^{2n}\big)\big(p_{z_R}^{(n)} \cos(n\, t_x)+ q_{z_R}^{(n)} \sin(n\, t_x)\big)\nonumber\\
			&\mkern-90mu - 2 R\int_{-\pi}^\pi  \widetilde{\del_{\nu_{y_R(t)}}\NOrk}(x_R,y_R(t))  \NORk(z_R,y_R(t))\intd t\nonumber\\
			&\mkern-90mu +\,\OO\Big(\frac{(R-r)^2}{r}\Big)\,.
	\end{align}
	}
	We can further simplify this approximation by using Green's identity on $\NOrk(z_r,x_r)$, with $\NOk(z_r,x_r)$, and using Taylor series on $\NOk(z_R,x_R)$ and on $\widetilde{\NOk}$. This leads us to the equation
	{\small 
	\begin{align*}
		2 \NORk(z_R,x_R) 
			=& \,2\,\NOrk(z_r,x_r)+2(\NOk(z_R,x_R)-\NOk(z_r,x_r))\nonumber\\
			&\mkern-90mu -2r\,\int_{-\pi}^\pi \big(\del_{\nu_{y_r(t)}} \widetilde{\NOk}(x_R,y_r(t))-\del_{\nu_{y_r(t)}} \widetilde{\NOk}(x_r,y_r(t))\big)\NOrk(z_r,y_r(t))\intd t\nonumber\\
			&\mkern-90mu -2r\int_{-\pi}^\pi \del_{\nu_{y_r(t)}} \widetilde{\NOk}(x_r,y_r(t)) \, 
				\tfrac{1}{2\pi}\log\bigg(\frac{\tfrac{R^2+r^2}{2\,R\,r}-\cos(t_z-t)}{1-\cos(t_z-t)}\bigg)
			\intd t\nonumber\\
			&\mkern-90mu  -\sum_{n=1}^\infty\big(1- (\tfrac{r}{R})^n\big)\big(p_{z_r}^{(n)} \cos(n\, t_x)+ q_{z_r}^{(n)} \sin(n\, t_x)\big)\nonumber\\
			&\mkern-90mu +\sum_{n=1}^\infty\big(1- (\tfrac{r}{R})^{2n}\big)\big(p_{z_R}^{(n)} \cos(n\, t_x)+ q_{z_R}^{(n)} \sin(n\, t_x)\big)\\
			&\mkern-90mu - 2 R\int_{-\pi}^\pi  \widetilde{\del_{\nu_{y_R(t)}}\NOrk}(x_R,y_R(t))  \NORk(z_R,y_R(t))\intd t\\
			&\mkern-90mu +\,\OO\Big(\frac{(R-r)^2}{r}\Big)\,.
	\end{align*}
	}
	Using that $\widetilde{\del_{\nu_{y_R(t)}}\NOrk}(x_R,y_R(t))=\OO(R-r)$, $\big(1- (\tfrac{r}{R})^n\big)=\OO(R-r)$ and that the logarithm in the second integral is in $\OO_{L^2}(R-r)$, we can infer that $\NORk(z_R,x_R) -\,\NOrk(z_r,x_r)=\OO(R-r)$ and thus make further simplifications, which lead to Equation (\ref{equ3:thm:RmrApprox}) and finishes the proof.
\end{proof}

\begin{remark}\label{rem:1}
	We note here, that Equation (\ref{equ9:thm:RmrApprox}) is numerically more stable than Equation (\ref{equ3:thm:RmrApprox}) in Theorem \ref{thm:RmrApprox}. We expect the reason to be that the constant error in the first step is lowered by the factor $(R-r)$ and the factor $1/2$ in every subsequent step.
\end{remark}

\section{Numerical Implementation and Application}\label{sec:NumImplTest}

\subsection{Applying Theorem \ref{thm:RmrApprox} - Gradually Increasing the Radius}\label{sec:NumImpl:RmrIncrease}
With Theorem \ref{thm:RmrApprox} we are able to evaluate the Neumann function $\NOrk(z_r,x_r)$ while we increase the radius of the circular sub-domain $B_r(\zeta)$ in $\Omr$ by $\Delta r$, where $z_r,x_r \in \del B_r(\zeta)$, with an error in $\OO((R-r)^2)$. Similar to how we numerically evaluate the solution to an ordinary differential equation $y(t) = f(t, y(t))$, $y(t_0)=y_0$, using the explicit Euler scheme, where we start at $t_0$ and then evaluate the function $y$ at $t_0+\Delta t$ with an error in $\OO((\Delta t)^2)$, we can now evaluate the function $\NOrk(z_r,x_r)$ at radius $R = r+\Delta r$. For the Euler scheme, we can show using Grönwall's inequality that the global error is  $\OO(\Delta t)$. Thus we expect the global error of $\NOrk(z_r,x_r)$ to be  $\OO(\Delta r)$.

The domain $\Omr$ for the numerical evaluation is set to be $\Omr = B_r([0,0]^\mathrm{T})\cup B_1([1,2.5]^\mathrm{T})$. We increase the radius of $B_r$ in $\Omr$ by $\Delta r$ successively until the radius reaches $1$. In every step we compute the first $N_f$ Fourier coefficients of the smooth part of $\NOrk(z_r, \,\cdot\, )$, see (\ref{SmoothFourierExp}), using Theorem \ref{thm:RmrApprox} with Remark \ref{rem:1}, where we also have to discretize $z_r$ in such a way that we have $N_f$ equidistant points on $\del B_r(0)$, where one point is set at $[-r,0]^\mathrm{T}$. For the first step we use  Equation (\ref{equ3:thm:NOrk-Asymptotics}) in Theorem \ref{thm:NOrk-Asymptotics}. 

\begin{figure}[h]
  \centering
  \includegraphics[width=0.99\textwidth]{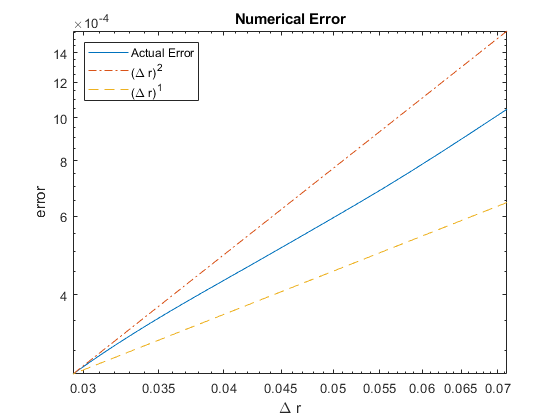}
  \caption{We have depicted the average error between the smooth enough Neumann function $\widetilde{\mathrm{N}_{\Om}^k}(z_1,x_1)$ and the numerical approximation with respect to the radius increase $\Delta r$, at the $N_f\in\NN$ points $z_1,x_1 = 1\cdot[\cos(t_n), \sin(t_n)]^\mathrm{T}$, $t_n = -\pi, ... , \pi$. We see that the error is at least linear in $\Delta r$. We set here $k=1$, $N_f=2^8$.}\label{fig:InflationErrorPlot}
\end{figure}

In Figure \ref{fig:InflationErrorPlot} we have depicted the error, which we calculated using MATLAB, between the actual Neumann function and the approximation given through the algorithm corresponding to $\Delta r$. To be more precise, we computed all possible $N_f^2$ discretized values of the smooth enough part of $\mathrm{N}_{\Om_1}^k(z_1,x_1)$ for $z_1,x_1\in \del B_1(\zeta)$ and averaged them in the numerical approximation. The actual Neumann function was numerically computed using the BEM with a very large number of boundary points. The Figure shows that we indeed achieve an error in $\OO((\Delta r)^1)$. It seems that we even achieve a higher order than only a linear one, but this is not further investigated here.

Comparing this numerical approximation with the BEM, we see that for this approximation we have have a runtime complexity of $\OO(N_f^2)\times \OO((\Delta r)^{-1})$ and an error in $\OO(\Delta r)$ multiplied to an error with respect to $N_f$, which in the above numerical experiments had no influence. For the BEM we have to invert a $N_c \times N_c$ matrix, where $N_c$ is the amount of discretisation points used on the boundary, which yields an error in $\OO(N_c^{-1})$ and has a complexity runtime of $\OO(N_c^3)$ in simple algorithms.

\subsection{Reconstructing a Matrix}\label{sec:NumImpl:FindingA}
In this section we use the approximation shown in the last section to determine a specific scattering matrix $\mathbf{S}$, as it is elaborated in Section \ref{sec:Prelim}. Different than in Equation (\ref{equ:S = N(A-I)}) we search here for a matrix $\mathbf{S}$, which is as close as possible in average value to all entries to a predetermined Matrix, which we call in this subsection matrix $\mathbf{A}\in\CC^{N\times N}$. Thus we try to minimise the value $e(\mathbf{S}) = \frac{1}{N^2}\sum_{i, j}|A_{i,j}-S_{i,j}|$. 

The procedure to form such a matrix $\mathbf{S}$ is as follows. We have $N$ source points $(z_i)_{i=1}^N$ equidistantly distributed in $(0,1)\times \{0\}$. When there are no scattering objects placed in $\RR^2_+$, then the Neumann function $\NOk$ is simply the $\Gamma^k$ function, and hence $S_{i,j} = \del_{(z_j)_2} \NOk(z_i,z_j)= 0$. Next we place a small ball within $\RR^2_+$, where we place the center so that the error $e(\mathbf{S})$ is minimised, which we in turn calculate using Theorem \ref{thm:NOrk-Asymptotics}. We did this minimization classically using a grid of points, but can in general be realized with more sophisticated methods as for example with the gradient descend method. Given the initial ball, we increase its radius using Theorem \ref{thm:RmrApprox} as it is shown in the previous section. After every increase we compute the Neumann function at the source points using the associated integral formulation, that is, 
\begin{align}
	\NOrk(z_i,z_j) 
		&= \NOk(z_i,z_j) - \int_{\del B_r(\zeta)}\del_{\nu_y} \NOk(z_j,y)\NOrk(z_i,y)\intd\sigma_y\,,\label{equ:NumImpl:1}\\
	\NOrk(z_i,y) 
		&= \NOk(z_i,y) - \int_{\del B_r(\zeta)}\del_{\nu_w} \NOk(z_j,w)\NOrk(w,y)\intd\sigma_w\,.\label{equ:NumImpl:2}
\end{align}
Thus we can compute $e(\mathbf{S})$, with the objective to see whether the error decreases or increases and whether we should increase the radius further or not. As soon as an increase in the radius does not yield a lower error, we search for a place to add another small ball. We again use Theorem \ref{thm:NOrk-Asymptotics} to determine the next best place to center the ball. In addition, we need to calculate $\NOk(\zeta_i,\zeta_j), \nabla \NOk(\zeta_i,\zeta_j) , (\nabla_{\zeta_i}\nabla_{\zeta_j}^\mathrm{T}) \, \NOk(\zeta_i,\zeta_j)$ in order to apply Theorem \ref{thm:RmrApprox}, where $\zeta_i, \zeta_j$ are values in $\RR^2_+\setminus \Omr$ and where $(\nabla_{\zeta_i}\nabla_{\zeta_j}^\mathrm{T})$ denotes a $2 \times 2 $ matrix in which the entries are the respective coordinate differentiations. To this end, we use the integral formulation above, in which we can interchange integration and differentiation. In practice, we used a linear interpolation to speed up the calculation. After we established a new place for the small ball, we can also increase it until the error $e(\mathbf{S})$ does not decrease any further. And then we search for a place for a third ball, and then a fourth and so forth until we cannot decrease $e(\mathbf{S})$ any further. This algorithm is explicit and does not use the inversion of any matrix as it is commonly done using a BEM. 

\begin{figure}[h]
  \centering
  \includegraphics[width=0.99\textwidth]{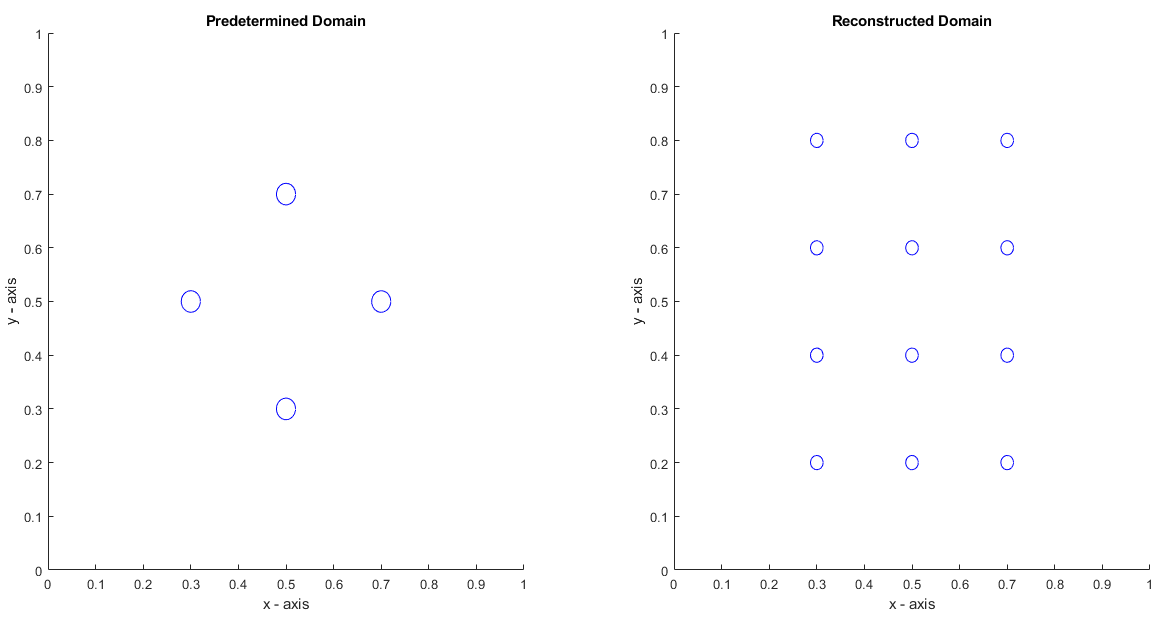}
  \caption{On the left side  a domain which yields a specific scattering matrix is given. On the right we have the domain given by the algorithm. In Figure \ref{fig:AlgoPredHeatMap} we can examine both scattering matrices. The predetermined domain is build of 4 circles of radius $0.02$ with center $[0.5, 0.3]^\mathrm{T}$, $[0.7, 0.5]^\mathrm{T}$, $[0.5, 0.7]^\mathrm{T}$, $[0.3, 0.5]^\mathrm{T}$.}\label{fig:AlgoPredDom}
\end{figure}

For our first numerical experiment, we set our predetermined matrix $\mathbf{A}$ to be the scattering matrix of a predetermined domain, which is given on the left-hand side in Figure \ref{fig:AlgoPredDom}. Using the algorithm described above, we obtain the domain on the right-hand side. On the left-hand side in Figure \ref{fig:AlgoPredHeatMap} we see a heat-map of the real part of the matrix $\mathbf{A}$ and on the right-hand side we see a heat-map of the real part of the scattering matrix $\mathbf{S}$.

\begin{figure}[h]
  \centering
  \includegraphics[width=0.99\textwidth]{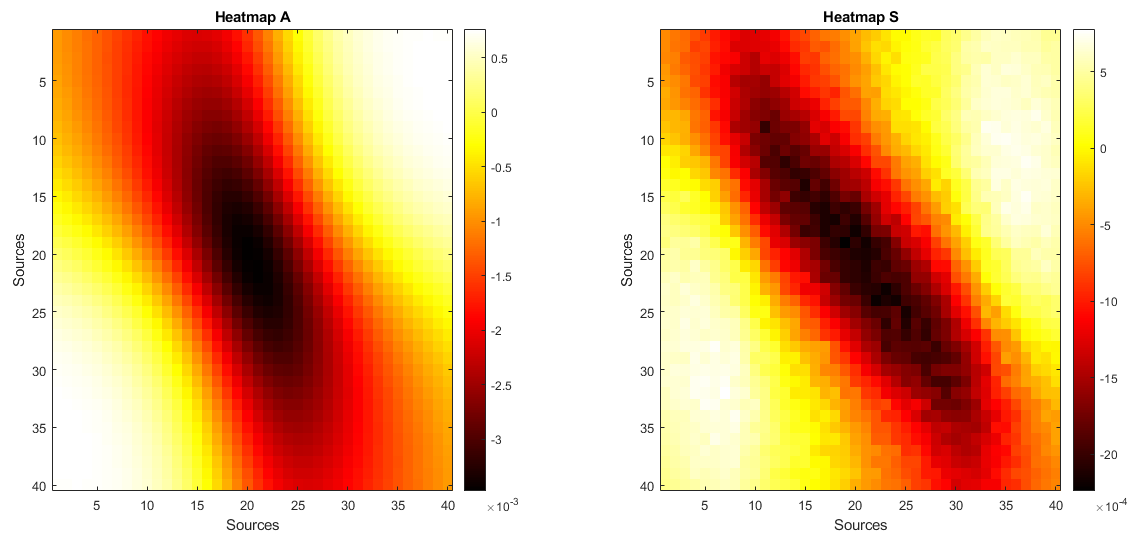}
  \caption{Here we see the real part of the corresponding scattering matrices to Figure \ref{fig:AlgoPredDom}. On the right-hand side we have the approximated one given by the algorithm. The sources are enumerate from $1$ to $N$, where the first corresponds to the leftmost source on the x-axis.}\label{fig:AlgoPredHeatMap}
\end{figure}

For more general matrices $\mathbf{A}$ we need more sources than given by the size of $\mathbf{A}$. To have such a more extended matrix we have to cast $\mathbf{A}$ to a integral of the form $\int_{\Lambda}\mathrm{K}(z_i,y)u(y)\intd\sigma_y$, and finally discretize that integral, and then apply the algorithm to the discretization.

\section{Concluding Remarks}\label{sec:ConcludingRemarks}
We considered the physical experiment presented in \cite{EnghetaPaper}, in which scattering objects were placed in front of signal sources. Those sources send waves which reflect at the object and then receiving points collect the wave intensity. The registered intensity is the solution to a predetermined linear system of equations. Hence, instead of solving the linear system with mathematical means, we can solve it using a physical set-up, which is substantially faster. The complication arises in finding the exact configuration of the scattering objects.

Using a mathematical model for the underlying physical problem we were able to describe the PDE using the Neumann function. Studying its asymptotic behaviour when we place tiny scattering objects and when we increase the extent of those objects successively, we were able to develop an explicit algorithm to place and enlarge objects such that the scattering matrix approaches the predetermined matrix, which is needed to solve the linear system of equation. In Section \ref{sec:NumImplTest} we showed that the numerical implementation for calculating the Neumann function when we enlarge an object works better then intended, in regard of the explicit Euler scheme. With such an algorithm we have a new and faster numerical method to calculate the Neumann function than using the ordinary BEM. We then applied that process to approach a desired matrix.

In this paper we considered circular scattering objects. It would be interesting to have more complicated domains such as ellipses, which would allow for one more easily accessible degree of freedom to control the waves. We think that the mathematical proofs in Section \ref{Sec:AsymptoticFormula} can be readily extended to more complicated $\cC^2$-boundaries. To this end, we need to consider a function $\phi: (-\pi,\pi)\rightarrow \RR^2$, which described the boundary, and consider it in the integration formulae. 

In the last section we mentioned that reconstructing a more general matrix $\mathbf{A}$ in a linear system of equation does not work well. We need more options in our algorithm, or a bigger matrix, which has similar properties to $\mathbf{A}$, and additionally can be described as a kernel of an integration operator. In \cite{EnghetaPaper}, the authors set the matrix to be the lower left quadrant of their scattering matrix. 

We are looking forward to see these asymptotic formulae being used in other physical problems concerning scattering problems. We are also very curious to see improvements in the object reconstruction of general linear systems and hope that our research will lead to an improvement of mathematical and technological tools for numerical computing.

\appendix
\section{An Integral Identity}\label{app:A}
In this appendix we derive the following identity:
{\small 
\begin{align*}
	- r \int_{-\pi}^\pi   
		\frac{-1}{2 \pi}\frac{2 R r-(R^2+r^2)\cos(t_x-t)}{(R^2+r^2-2 R r \cos(t_x-t))^2} \, 
		&\frac{1}{2 \pi}\log\Big(\frac{R^2+r^2}{2\,R\,r}-\cos(t_z-t)\Big)\intd t\\
	=&\,\frac{r^2}{2\pi\, R}\frac{R^2\cos(t_z-t_x)-r^2}{R^4+r^4-2 R^2 r^2\cos(t_z-t_x)}\,.
\end{align*}
}
Using the $2\pi$ periodicity, we can rewrite the left-hand side in the above identity as 
{\small 
\begin{align*}
	- r \int_{-\pi}^\pi   
		\frac{-1}{2 \pi}\frac{2 R r-(R^2+r^2)\cos(t-\tau)}{(R^2+r^2-2 R r \cos(t-\tau))^2} \, 
		\frac{1}{2 \pi}\log\Big(\frac{R^2+r^2}{2\,R\,r}-\cos(t)\Big)\intd t\,,
\end{align*}
}
where $\tau\DEF t_x-t_z$. Then we use the Fourier series
{\small 
\begin{align*}
	\frac{1}{2 \pi}\log\Big(\frac{R^2+r^2}{2\,R\,r}-\cos(t)\Big)
		= \frac{1}{2 \pi}\log\Big(\frac{R}{2\,r}\Big)
		- \frac{1}{\pi}\sum_{n=1}^\infty\frac{(\nicefrac{r}{R})^n}{n}\cos(n\,t)\,,
\end{align*}
}
and subsequently the following identity
{\small 
	\begin{align*}
		\int_{-\pi}^\pi  \frac{-1}{2 \pi}\frac{2 R r-(R^2+r^2)\cos(t_x\!-\!t)}{(R^2+r^2-2 R r \cos(t_x\!-\!t))^2} \cos(n (t_z\!-\! t))\intd t=\frac{n}{2 R r}\Big(\frac{r}{R}\Big)^n\!\!\cos(n(t_z\!-\! t_x))\,,
	\end{align*}}
for all $n\in \NN_0$, to obtain that
{\small 
\begin{align*}
	\frac{1}{2 \pi\, R }\sum_{n=1}^\infty  
		\big(\tfrac{r}{R}\big)^{2n}\cos(n\tau)\,.
\end{align*}
}
This infinite sum is the Fourier sum of
{\small 
\begin{align*}
	\frac{1}{2 \pi\, R }\frac{r^2 (R^2\cos(\tau)-r^2)}{R^4+r^4-2 R^2r^2\cos(\tau)}\,,
\end{align*}
}
which is the desired term.

\section{Modification to the Trapezoidal Rule}

In Section \ref{sec:Prelim}, we need to calculate the integral
$
	\int_{\Lambda} u_j(x)\,\del_{x_2} \NOk(z_i,x) \intd\sigma_x \,,
$
using the trapezoidal rule. But the function $u_j(x)$, where $x=[x_1, x_2]^\mathrm{T}$, $x_1\in (0,1)$, $x_2 = 0$, is not well defined for $x = z_i$. It has a logarithmic singularity around $z_i$. To use the trapezoidal rule, we need to modify it slightly. Let us be more general and consider an integral of the form
\begin{align*}
	\int_0^1 \log(|t-t_\ast |)\,f(t)\intd t \,,
\end{align*}
where $t_\ast\in (0,1)$ and $f:[0,1]\rightarrow\CC$ is a twice continuously differentiable function. Assume we have $N$ strictly increasing grid points $t_1=0,\ldots,t_N=1$, where $t_m =t_\ast$. We define $\Delta_i = t_{i+1}-t_i$. Then we have that
{\small
\begin{align*}
	\int_{t_m}^{t_{m+1}}&  \log (|t-t_\ast |)\,f(t)\intd t  \\
		&= [f(t)\,\big((t-t_\ast)\log(t-t_\ast)-t\big)]_{t=t_m}^{t_{m+1}} - \int_{t_m}^{t_{m+1}}\!\!\! f'(t)\,\big((t-t_\ast)\log(t-t_\ast)-t\big)\intd t\,,\\
		&= \tfrac{1}{2}f(t_{m+1})\Delta_m (\log(\Delta_m)-2)+\tfrac{1}{2}f(t_{m+1})\log(\Delta_m)\Delta_m +\OO((\Delta_m)^2\log(\Delta_m))\,,
\end{align*}}\noindent
where we used partial integration in the first equation, and in the second one that $f(t_{m+1})=f(t_{m})+\OO(\Delta_m)$ and $f'(t)=(f(t_{m+1})-f(t_{m}))/\Delta_m +\OO(\Delta_m)$. Similarly, we have that
{\small
\begin{align*}
	\int_{t_{m-1}}^{t_m}  \log (|t-t_\ast|)\,f(t)&\intd t  
		= \tfrac{1}{2}f(t_{m-1})\Delta_{m-1} (\log(\Delta_{m-1})-2)\\
		&+\tfrac{1}{2}f(t_{m-1})\log(\Delta_{m-1})\Delta_{m-1} +\OO((\Delta_{m-1})^2\log(\Delta_{m-1}))\,.
\end{align*}}\noindent
Now we define $(f_i)_{i=1}^N=(f_i(t_i))_{i=1}^N$ and $(l_i)_{i=1}^N=(\log(|t_i-t_\ast|)_{i=1, i\neq m}^N$, with $l_m = (\log(\Delta_m)+\log(\Delta_{m-1})-4)/2$. We have then
{\small
\begin{align*}
	\int_0^1 \log(|t-t_\ast |)\,f(t)\intd t = \tfrac{1}{2}f_1\,l_1\,\Delta_1+\sum_{i=2}^{N-1} f_i\,l_i \, \Delta_i+\tfrac{1}{2}f_N\,l_N\,\Delta_N+\OO(\max_{i=1,\ldots,N} \!\!\Delta_{i}|\log(\Delta_{i})|)\,.
\end{align*}}\noindent

\bibliographystyle{plain}
\bibliography{paperInverseMetaStruct}

\end{document}